%% file: LpLq_approach.tex
\documentclass[11pt, twoside, reqno]{article}

\include{preamble}
\usepackage[normalem]{ulem}
\begin{document}
\title{Blow-up criterion for the compressible Navier--Stokes system with inflow-outflow boundary conditions}
\author{Anna Abbatiello$^a$ \and Mostafa Meliani$^b$\footnote{Corresponding author. Email: \href{mm4138@bath.ac.uk}{mm4138@bath.ac.uk}} }
\date{}
\maketitle

\centerline{$^a$Department of Mathematics and Physics, University of Campania ``L. Vanvitelli''}

\centerline{Viale A. Lincoln 5, 81100 Caserta, Italy}

\medskip
\medskip

\centerline{Department of Mathematical Sciences, University of Bath, Bath, BA2 7AY, UK}
 
\medskip

\begin{abstract}
 We consider the compressible Navier-Stokes system in three dimensions with general inflow-outflow boundary conditions, meaning that we prescribe a boundary velocity which has a negative normal component and accordingly the density is prescribed on the inflow part of the boundary.  We establish a blow-up criterion in a class of strong solutions  in the $L^p-L^q$ framework. In particular assuming the boundedness of the quantities $(\varrho^{-1}, \bu)$ and of a suitable norm of $\nabla_x \varrho$ the solution remains regular and the blow-up does not occur. We develop the condition on $\nabla_x \varrho$ because we need a new approach in order to accommodate the inhomogeneous boundary conditions, as the standard estimates on the material time derivative work when the normal component of the boundary velocity is zero.
\end{abstract} 
\medskip



\section{Introduction}

This paper is concerned with a blow-up criterion for the isentropic compressible Navier-Stokes system in a three dimensional space domain with inhomogeneous boundary conditions. The Navier-Stokes system of equations describes the time evolution of the mass density $\varrho=\varrho(t,x)$ and the velocity $\bu=\bu(t,x)$ of a viscous compressible fluid occupying a bounded space domain $\Omega\subset\mathbb{R}^3$ during a time interval $(0,T)$ and it reads as follows
\begin{align}\label{NS:system_of_equations}
\partial_t \varrho + \div_x (\varrho \bu) &= 0,
\\
\partial_t(\varrho \bu) + \div_x (\varrho \bu \otimes \bu) +\nabla_x p(\varrho) & = \div_x \mathbb{S}(\nabla_x\bu) + \varrho \mathbf{g}.
\end{align}
The pressure is given through the relation $p(\varrho) = a \varrho^\gamma$ with $a>0$ and $\gamma>1$, while the viscous tensor $\SS(\nabla_x\bu)$ is given 
by the Newton's law
\begin{equation}\label{eq::viscosity_thing}
	\SS(\nabla_x\bu) =  2\mu \mathbb{D}_x\bu 
    + \lambda \div_x \bu \mathbb I
\end{equation}
with constant viscosity coefficients $\mu, \lambda$, satisfying $\mu\geq0$ and $\lambda+ \frac23 \mu\geq 0$. Also $\mathbb{D}_x \bu =\frac12 \nabla_x \bu +\frac12 (\nabla_x\bu)^t$ is the symmetrized gradient and $\mathbb{I}$ the identity matrix in $\R^3$. The system is endowed with initial conditions
$$\varrho(0, \cdot)=\varrho_0, \ \bu(0, \cdot)=\bu_0.$$
Finally, we impose a Dirichlet type boundary condition for the velocity and accordingly the mass inflow boundary condition
\begin{equation}\label{boundary}
    \bu_{|\partial \Omega}=\bu_B, \ \varrho_{|\Gamma_{\rm in}}=\varrho_B, \  \Gamma_{\textup{in}}=\{x\in \partial\Omega: \bu_B\cdot \mathbf{n}<0\}
\end{equation}
where $\bu_B=\vu_B(x)$ is a given field independent of time. 
We point out that we allow an inflow boundary, in the sense that in what follows $\Gamma_{\textup{in}}$ may be, in general, nonempty. \\
In the present work, we study the conditions that guarantee the global existence of strong solutions to the open system \eqref{NS:system_of_equations}--\eqref{boundary}. 
The local well-posedness of the open compressible Navier-Stokes system in the
$L^p-L^q$ class follows from the result established in \cite{meliani2025p}, where in particular the full Navier-Stokes system coupled to a magnetic induction model is considered. \\
Given the data in the underlined class, our aim is to establish estimates involving lower order norms implying that the local strong solution remains regular as soon as the identified quantities are bounded. Therefore the maximal existence time of the solution is either infinite or the recognized quantities blow-up.
Blow-up criteria for the compressible Navier-Stokes system with inhomogeneous boundary conditions have been studied in \cite{abbach25,abfe24, basaric2023conditional}, where they all assume that $\bu_B\cdot\mathbf{n}=0$. This assumption is crucial in the analysis. It also implies that the material time derivative of $\bu-\bu_B$ is zero on the boundary of the spatial domain. Both conditions are exploited in several parts when considering the estimates for the material time derivative as in the original spirit of Hoff \cite{hoff1995global}. 
Having a general boundary velocity implies that the estimate on the velocity material derivative may be achieved requiring that $\nabla_x \bu$ is bounded on the boundary, but then the higher order velocity material derivative estimate may be out of reach because of the need to deal both with $\nabla_x\varrho$ and with the gradient of the velocity material derivative on the boundary. We also explored the possibility to modify the material derivative in order to keep the property $D_t(\bu-\bu_B)_{|\partial\Omega}=0$ (where the symbol $D_t$ stands for the material derivative), but we ended with the need of controlling the second gradient of the velocity when computing the higher order material derivative estimate, which seems to be out of reach at that level. 
Therefore, we tackle the problem by invoking the $L^p-L^q$ regularity theory developed by Denk-Hieber-Pr\"{u}ss \cite{denk2007optimal} for parabolic initial-boundary value problem. In particular, we split the velocity into two parts: one solving the Lam\'{e} system and the remaining solving a parabolic problem with right-hand side. 
Another essential part is the way we treat boundary integrals on $\Gamma_{\rm in}$. As in \cite{meliani2025p}, we obtain an a priori estimate for the normal derivative of $\varrho-\varrho_B$ on $\Gamma_{\rm in}$ assuming that $\bu_B\cdot\mathbf{n}$ is bounded away from zero.
Finally, as in \cite{abfe24},  formulating a blow-up criterion in the $L^p-L^q$ class of regularity is also optimal because we require less regular data and less compatibility conditions.
\\
To our best knowledge, this is the first result on a blow-up criterion for the open compressible Navier-Stokes system with general boundary data.\\
Other references concerning conditional regularity and blow-up criteria for the compressible Navier-Stokes and for the compressible Navier-Stokes-Fourier system are e.g. \cite{fang2012blow, sun2011bealeJMPA, sun2011bealeARMA, fan2008blow, fan}. The interested reader may also check the references therein.

\section{Problem formulation and main result} \label{sec:Main_result}
In this work, we need a few different function spaces, among which Sobolev spaces $W^{k,p}$, Besov spaces $B^{s}_{pq}$, Triebel--Lizorkin spaces $F^s_{pq}$. The use of Besov and Triebel--Lizorkin spaces is the result of a quest of optimality in the $L^p$-$L^q$ framework by Denk--Hieber--Pr\"uss~\cite{denk2007optimal}. The interested reader may find the introduction to function spaces by  Triebel~\cite{triebel1983theory,triebel1992theory} useful. 

We begin by presenting the assumptions on the domain and on the initial and boundary data. The assumption on the inflow part of the boundary is justified by the discussion in \cite[Section 5]{meliani2025p}.
\begin{assumption}[Assumptions on the domain]\label{assu:domain}
	Let $\Omega$ be a bounded domain of $\R^3$  with boundary $\partial \Omega$ of class at least $C^2$. We assume that the inflow part of the boundary $\Gamma_{\textup{in}}$ is a closed surface of non-zero Hausdorff measure, i.e., $\mathcal{H}^{2} (\Gamma_{\textup{in}})\neq 0$.
\end{assumption}

\begin{assumption}[Assumptions on the boundary data]\label{assu:data}
	Let $p,q \in (1,\infty)$ satisfy 
    \[q>3, \qquad \max\left\{\frac{2q}{q-1}, \frac{2q}{2q-3}\right\}<p\] and let $T>0$. We assume that the boundary and initial data for the density satisfy
	\begin{equation}\label{data_density}
		\begin{aligned}
            &  \varrho_0 \in W^{1,q}(\Omega), \ \varrho_0>0,\\
			&  \varrho_B \in W^{1,q}(0,T; L^q(\Gamma_\textup{in})) \cap L^{q}(0,T; W^{1, q}(\Gamma_\textup{in})), \ \varrho_B>0,\\
            & 
            \varrho_0(x) = \varrho_B(x,0)   \ \textrm{for a.e.}\ x \in \Gamma_{\textup{in}},
		\end{aligned}
	\end{equation}
while we assume for the velocity field that
		\begin{equation}\label{data_velocity}
		\begin{aligned}
        	&\bu_0 \in B^{2(1-1/p)}_{qp}(\Omega),\\
            &\bu_B \in W^{2-1/q,q}(\partial \Omega),\\
            & \bu_0(x) = \bu_B(x)  \quad \ \textrm{for a.e.}\ x \in \partial \Omega,\\
			& - \bu_B \cdot  \mathbf{n}\geq c >0   \quad \ \textrm{for some }\ c>0 \ \textrm{ on } \Gamma_{\textup{in}}.
		\end{aligned}
	\end{equation}
	Finally, we consider the force $g \in L^p(0,T; L^q(\Omega))$.
\end{assumption}
 Note that to focus onto the difficulties of the problem we assume the boundary velocity $\bu_B$ to be independent of time.

We ensure ellipticity of the operator $\div_x(\SS(\nabla_x \, \cdot\, ))$ by imposing that the viscosity coefficients satisfy
\begin{equation}
    \mu>0, \qquad \lambda+ \frac{2}{3}\mu\geq 0.
\end{equation}

Our main result pertains to the conditional global well-posedness of \eqref{NS:system_of_equations}  with general Dirichlet boundary conditions \eqref{boundary}. It states the conditions for the unique continuation of the local solution $(\varrho,\bu)$ in the $L^p$--$L^q$ class determined in \cite{meliani2025p}. Before, we state the existence result which can be proved following the main steps in \cite{meliani2025p}.
\begin{theorem}
Let Assumptions~\ref{assu:domain} and~\ref{assu:data} hold. Then there exists a strictly positive time $T > 0$ such that the system \eqref{NS:system_of_equations}-\eqref{boundary} has a unique
solution
\begin{equation}\label{eq:def-Lp-Lqclass}
(\varrho, \bu) \in L^\infty(0,T; W^{1,q}(\Omega)) \cap W^{1,\infty}(0,T;L^q(\Omega)) \times L^p(0,T; W^{2,q}(\Omega)) \cap W^{1,p}(0,T;L^q(\Omega)).
\end{equation}
\end{theorem}

The assumptions on $p,q$ are needed to show the local well-posedness of the MHD system~\cite{meliani2025p}. We do not underline any substantial difference in the proof.
\\
Our main theorem is the following:
\begin{theorem}\label{thm:main}
Let Assumptions~\ref{assu:domain} and~\ref{assu:data} hold. Let $(\varrho,\bu)$ be the local strong solution in the $L^p$--$L^q$ class \eqref{eq:def-Lp-Lqclass} and let $T^*$ be the first time such that the solution cannot be continued in this class to time $T=T^*$, then either $T^* =\infty$ or
\begin{equation}\label{blow_up_criterion}
\limsup_{t \to (T^*)^-} \left( \left\|\frac1 \varrho(t)\right\|_{L^\infty(\Omega)} + \|\bu(t)\|_{L^\infty(\Omega)} + \|\nabla \varrho(t)\|_{L^7(\Omega)} \right)=\infty.
\end{equation}
\end{theorem}

\begin{remark}[Blow-up criterion for $\varrho$] \label{rmk:no_varrho}
In this remark, we show that under the hypotheses of Assumptions~\ref{assu:domain} and \ref{assu:data}, the criterion~\eqref{blow_up_criterion} is equivalent to the blow-up criterion 
\[\label{blow_up_criterion_equiv}
\limsup_{t \to (T^*)^-} \left(\left\|\varrho(t)\right\|_{L^\infty(\Omega)} + \left\|\frac1 \varrho(t)\right\|_{L^\infty(\Omega)} + \|\bu(t)\|_{L^\infty(\Omega)} + \|\nabla \varrho(t)\|_{L^7(\Omega)} \right)=\infty,
\]
where the norm $\left\|\ \varrho(t)\right\|_{L^\infty(\Omega)}$ is present.

The implication \eqref{blow_up_criterion}$\implies$\eqref{blow_up_criterion_equiv} is obvious. The converse also holds. We prove it here by contradiction. 
To this end, we use \cite[Theorem 4.1.7]{hsiao2021boundary} which ensures, when $\mathcal{H}^2( \Gamma_\textup{in}) \neq 0$ (see Assumption~\ref{assu:domain}), that the following inequality holds:
\begin{equation}\label{equivalent_norm}
    \|\varrho(t)\|_{H^1(\Omega)} \leq C (\|\nabla \varrho(t)\|_{L^2(\Omega)} + \|\varrho_B\|_{L^2(\Gamma_\textup{in})}),
\end{equation}
for a.a. $t \in (0,T^*)$, for some $C>0$.
Since $\Omega$ is bounded:
\[\label{std_Lp_embdedding}
\|\nabla \varrho(t)\|_{L^7(\Omega)} \geq C_{r,\Omega} \|\nabla \varrho(t)\|_{L^r(\Omega)}, 
\]
for some $C_{r,\Omega}>0$, for all $1\leq r\leq 7$. Going back to \eqref{equivalent_norm} and using \eqref{std_Lp_embdedding} with $r=2$, we obtain that $\varrho \in L^{\infty}(0,T^*; H^1(\Omega))\hookrightarrow L^{\infty}(0,T^*; L^6(\Omega))$. We can then bootstrap our argument using \eqref{std_Lp_embdedding} with $r=6$, and conclude that  that 
\[\varrho \in L^{\infty}(0,T^*; W^{1,6}(\Omega))\hookrightarrow L^{\infty}(0,T^*; L^\infty(\Omega)).\]

This finishes the proof of the implication \eqref{blow_up_criterion_equiv}$\implies$\eqref{blow_up_criterion}.
\end{remark}

 The condition on the gradient ensures that the density (or more precisely its inverse $1/\varrho$) is continuous on the closure $[0,T]\times\overline\Omega$ which we need to apply the Denk--Hieber--Pr\"uss result \cite[Theorem 2.3]{denk2007optimal}. We note that assumptions on the boundedness of the gradient have been made in the context of studying the conditional regularity of compressible fluids; see, e.g., \cite{fan2008blow} and the references contained in \cite{fang2012blow}.

 We point out here that while we need to include the gradient as one of the components for the blow-up criterion in Theorem~\ref{thm:main}, the target space for the density can be arbitrarily more integrable,  by picking $|\nabla \varrho|^{q}$ with $q>7$ in \eqref{eq:def-Lp-Lqclass}.

\section{Auxiliary theoretical results}
\subsection{The Lam\'e system}\label{sec:lame_reg_results}
We provide some necessary estimates for the Lam\'e system. The results provided here were proved in e.g. \cite{agmon1959estimates}, \cite[Section 2]{sun2011bealeJMPA}, and \cite{acquistapace1992bmo}.

Consider the following abstract boundary value problem for the Lam\'e operator $L:= \mu \Delta_x + (\lambda + 2/3 \mu)\nabla_x\div_x $ with a general right-hand side $F$:
\begin{equation}\label{eq:Lame_abstract}
    \left\{
    \begin{array}{ll}
        L\,U = \div_x \mathbb{S}(\nabla_x U) = F & \quad \textrm{in } \Omega, \\
      U(x) = 0 &\quad \textrm{on } \partial\Omega.
    \end{array}
    \right.
\end{equation}
Here $U = (U_1,U_2,U_3)$ and $F = (F_1,F_2,F_3)$.  

\begin{proposition} \label{lemma:Lamme_Lp_BMO} Let $\Omega$ be a bounded Lipschitz domain and let $q \in (1,\infty)$. Let $U$ be a solution of \eqref{eq:Lame_abstract}. There exists a constant $C$ depending on $q$ and $\Omega$ such that the following estimates hold:
\begin{enumerate}[label=\arabic*)]
    \item If $F \in L^q(\Omega)$, then
    \begin{equation}\label{ineq:strong_ell_estimate}
            \|U\|_{W^{2,q}(\Omega)} \leq C \|F\|_{L^q(\Omega)}. 
    \end{equation}
    \item If $F \in W^{-1,q}(\Omega)$ (i.e. $F=\div_x g$ with $g = (g_{ij})_{3\times 3}$, $g_{ij} \in L^q(\Omega)$), then
    \begin{equation} \label{ineq:weak_ell_estimate}
            \|U\|_{W^{1,q}(\Omega)} \leq C \|g\|_{L^q(\Omega)}. 
    \end{equation}
        \item If $F =\div_x g$ with  $g_{ij} = \partial_k h_{ij}^k, \, h_{ij}^k \in L^q(\Omega),$ $i,j,k = 1,2,3$, then
    \begin{equation} \label{ineq:ultra_weak_ell_estimate}
            \|U\|_{L^{q}(\Omega)} \leq C \|h\|_{L^q(\Omega)}. 
    \end{equation}
    \item If $F =\div_x g$ with $g = (g_{ij})_{3\times 3}$, $g_{ij} \in L^\infty(\Omega)$, then
    \begin{equation} \label{ineq:weak_BMO_estimate}
            \|\nabla U\|_{BMO(\Omega)} \leq C \|g\|_{L^\infty(\Omega)}. 
    \end{equation}
\end{enumerate}
\end{proposition}
\begin{proof}
    The reader is referred to \cite[Proposition 2.1 and 2.2]{sun2011bealeJMPA} for a proof. The statement of \eqref{ineq:ultra_weak_ell_estimate} in \cite{sun2011bealeJMPA} assumes $h_{ij}^k \in W^{1,q}_0(\Omega)$ but the statement carries over to $h_{ij}^k \in L^q(\Omega)$ by density of $W_0^{1,q}(\Omega)$ in $L^q(\Omega)$.
\end{proof}

\subsection{A logarithmic inequality: Brezis--Gallouet--Wainger}
The first instance of use of a logarithmic Sobolev inequality we are aware of is due to Brezis and Gallou\"et~\cite{brezis1980nonlinear} which they exploited to show the global existence of solutions to a cubic nonlinear Schr\"odinger equation. Later Brezis and Wainger proved a number of logarithmic inequalities for (so-called) limiting cases of Sobolev embeddings~\cite{brezis1980note}. 

In this paper, we use a version of these logarithmic inequalities due to \cite{sun2011bealeJMPA} which allow us to use the BMO bound from Lemma~\ref{lemma:Lamme_Lp_BMO} in order to guarantee a bound in $L^\infty(\Omega)$ at only a logarithmic cost in $W^{1,q}(\Omega)$, as follows.
\begin{lemma}\label{lemma:log_ineq}
Let $\Omega$ be a bounded Lipschitz domain and  $f \in W^{1,q}(\Omega)$ with $q\in (3,\infty)$. There exists a constant $C$ depending on $q$ and the Lipschitz property of $\Omega$ such that
\begin{equation}
    \|f\|_{L^\infty(\Omega)} \leq C \left(1 + \|f\|_{BMO(\Omega)} \ln(e+\|\nabla f\|_{L^q(\Omega)})\right)
\end{equation}
    
\end{lemma}
\begin{proof}
    The reader is referred to \cite[Lemma 2.3]{sun2011bealeJMPA} for a proof.
\end{proof}

Note that the fact that the (semi)norm in $W^{1,q}(\Omega)$ contributes only logarithmically allows us to use Gronwall's lemma. We refer the reader to \cite{brezis1980nonlinear} or \cite[Theorem 4.1]{meliani2024well} for how to apply Gronwall's lemma in this case. 

Note that employing the Sobolev embedding $W^{1,q}(\Omega) \hookrightarrow L^\infty(\Omega)$ instead of Lemma \ref{lemma:log_ineq}, the superlinear polynomial growth in the resulting inequality for the mass density would make impossible to rule out blow-up of the density. This is why it is important to use the logarithmic inequality from Lemma~\ref{lemma:log_ineq}.

\section{Velocity splitting}
There exist various strategies to split the velocity for the purpose of studying conditional existence for fluid flow systems; see, e.g., \cite{basaric2023conditional,sun2011bealeARMA,sun2011bealeJMPA}. The common idea behind the velocity splitting strategy is that the main difficulty comes from controlling the term $\nabla p(\varrho)$. By relying on splitting the velocity $\bu = \bw+\bv $ such that for
each $t\in [0,T)$, $\bv(t)$ solves the Lam\'e system:
\begin{equation} \label{eq:Lame_system}
\left\{\begin{array}{ll}
      L\,\bv: = \div_x \mathbb{S}(\nabla_x \bv) = \nabla p (\varrho), & \quad \textrm{in } \Omega, \\
      \bv(x) = 0, &\quad \textrm{on } \partial\Omega, 
\end{array}
\right.
\end{equation}
we are able to use available results for the existence of weak solutions for \eqref{eq:Lame_system} which only rely on the fact that $p(\varrho) \in L^p(\Omega)$ with $ p \in (1,\infty]$; see Section~\ref{sec:lame_reg_results} for precise statements related to the regularity of solution of \eqref{eq:Lame_system}. In what follows, we use for convenience the notation $L^{-1} f$ to refer to the solution of \eqref{eq:Lame_system} with right-hand side $f$.

Following the approach of \cite[Section 3]{sun2011bealeJMPA}, $\bw$ solve the parabolic initial-boundary value problem:
\begin{equation}
\label{eq:parabolic_equation_w}
    \left\{\begin{array}{ll}
    \pdt \bw - \frac1\varrho \div_x(\SS (\nabla_x \bw)) = F, & \quad \textrm{in } \Omega, \\
    \bw(t,x) = \bu_B, &\quad \textrm{on } [0,T)\times \partial\Omega,\\ 
    \bw(0,x) = \bw_0(x), &\quad \textrm{on } \Omega, 
    \end{array}
\right.
\end{equation}
where $\bw_0(x) = \bu_0(x) - L^{-1}(\nabla p(\varrho_0))$ and $F= -u \cdot \nabla u - L^{-1} \nabla\left(\pdt p(\varrho)\right) + g $. 
Expressing the problem for $\bw$ as a parabolic problem, and not as an elliptic problem, allows us to bypass the need to derive estimates for the material derivative ($D_t \bu := \pdt \bu + \bu\cdot \nabla \bu$) as done in, e.g., \cite{basaric2023conditional,sun2011bealeARMA,hoff1995global}. Indeed, for general Dirichlet boundary condition $\bu_B$, providing control on $\nabla D_t \bu$ seems out of reach; the proof of \cite{sun2011bealeARMA} relies on $\bu_B =0$, the estimates in \cite{basaric2023conditional} make use of  $\bu_B\cdot\mathbf{n} =0$. Note that the control of the gradient $\nabla D_t \bu$ is needed in order to use the embedding $H^1(\Omega) \hookrightarrow L^q(\Omega)$ for some $q>3$.

In order to allow for general Dirichlet boundary conditions, we rely, in this work, on $L^p$--$L^q$ theory (see, e.g., \cite{denk2007optimal}) to provide a different argument for the conditional existence of strong solutions of \eqref{NS:system_of_equations}.

\section{Estimates for the velocity}

The current section is dedicated to the proof of the following lemma which establishes conditional regularity estimates for the velocity components $\bw$ and $\bv$:
\begin{lemma}\label{lemma:estimates_velocity}
Let Assumptions~\ref{assu:domain} and \ref{assu:data} hold. Let $T>0$ and assume there exist positive constants $\bar r, \bar\bu, \bar G$ such that
\begin{equation}
\begin{aligned}\label{cond_reg_crit}
\sup_{t\in (0,T)} \left\|\frac 1\varrho(t)\right\|_{L^\infty(\Omega)} \leq  \bar r,\qquad
\sup_{t\in (0,T)} \|\bu(t)\|_{L^\infty(\Omega)} \leq  \bar\bu, \qquad \sup_{t\in (0,T)}\|\nabla \varrho\|_{L^\infty(\Omega)} \leq \bar G ,
\end{aligned}
\end{equation}
for some $\varepsilon>0$.
Then, the following estimates hold:
	\begin{equation}\label{w_velocity_ineq}
		\begin{aligned}
			\int_0^T \big[\| \pdt \bw \|^p_{L^q(\Omega)} + & \|\bw\|_{W^{2,q}(\Omega)}^p \big]\dt \leq  \Lambda(\bar\bu, \varrho_0, \bu_0, \bu_B,\bar r, \bar G, p, q,T)
		\end{aligned}
	\end{equation}
    where $\Lambda$ is an increasing function of the arguments  which stays bounded for bounded arguments.  Also there hold
\begin{align}
   & \|\nabla \bv \|_{L^\infty(0,T;BMO(\Omega))} \lesssim C(\bar G, \varrho_B), \label{v_velocity_ineq} \\
    &\|\pdt \bv \|_{L^p(0,T;L^q(\Omega))} \lesssim C(\bar G, \varrho_B,\bar \bu, \gamma, p, q, T) \left(\varepsilon\| \bw \|_{L^p(0,T; W^{2,q}(\Omega))}^p + 1  \right). \label{v_velocity_time_der}
\end{align}
\end{lemma}
Before presenting the proof, let us note that due to Remark~\ref{rmk:no_varrho}, \eqref{cond_reg_crit} yields that there exist $\bar\varrho:=\bar\varrho(\bar G, \varrho_B)$, such that
\[\label{cond_reg_crit_varrho}
\sup_{t\in (0,T)} \|\varrho(t)\|_{L^\infty(\Omega)} \leq  \bar\varrho.\]
\begin{proof}
Estimate \eqref{v_velocity_ineq} is obtained using Lemma~\ref{lemma:Lamme_Lp_BMO} and in particular inequality \eqref{ineq:weak_BMO_estimate}. The rest of this section is dedicated to prove estimates~\eqref{w_velocity_ineq} and \eqref{v_velocity_time_der}.

In order to obtain estimates on the $\bw$ component of the velocity, we intend to apply Denk-Hieber-Pr\"uss result \cite[Theorem 2.3]{denk2007optimal} to the initial-boundary value problem~\eqref{eq:parabolic_equation_w}. In order to do this we need to verify that the assumptions of the aforementioned theorem hold. 

\subsection*{Initial and boundary data regularity}
The initial data are given by
\[\bw_0(x) = \bu_0(x) - L^{-1}(\nabla p(\varrho_0)).\]
We know that $\bu_0 \in B^{2(1-1/p)}_{qp}(\Omega)$ because of  Assumption~\ref{assu:data}. Also, since $\varrho_0 \in W^{1,q}$ then Proposition~\ref{lemma:Lamme_Lp_BMO} ensures that $L^{-1}(\nabla p(\varrho_0)) \in W^{2,q}(\Omega) \subset B^{2(1-1/p)}_{qp}(\Omega)$. 
Thus $\bw_0$ enjoys sufficient regularity to apply the $L^p$-$L^q$ theory of Denk--Hieber--Pr\"uss~\cite{denk2007optimal}.

The regularity of the boundary data is ensured by Assumption~\ref{assu:data}.

\subsection*{Regularity of the coefficients} 
In order to apply Denk--Hieber--Pr\"uss framework~\cite{denk2007optimal} we need to check the regularity of the coefficient $\frac{1}{\varrho}$ which appears in front of the term $\div_x(\SS(\nabla_x \bw))$ in \eqref{eq:parabolic_equation_w}. Note that it is not sufficient that $\frac{1}{\varrho}$ belongs to  $L^\infty(0,T; L^\infty(\Omega))$. Rather, to apply~\cite[Theorem 2.3]{denk2007optimal}, we need $\frac{1}{\varrho} \in C([0,T] \times \bar\Omega)$, which we ensure thanks to  the assumption that $\int_0^t \int_\Omega |\nabla \varrho|^{7} \leq \bar G$.

Indeed, we first extend the boundary data in the space-time cylinder $(0,T)\times \Omega$ 
by considering the solution to the elliptic equation:
	\begin{equation}
		\begin{aligned}
			 - \div_x(\SS (\nabla_x\bw_B)) &= 0 && \textrm{on }\   \Omega,\\
			\bw_B &= \bu_B && \textrm{on }\ \partial \Omega.
		\end{aligned}
	\end{equation}
with $\bw_B \in W^{2,q}(\Omega)$. Which is then used to reformulate problem~\eqref{eq:parabolic_equation_w} as a zero-Dirichlet boundary condition problem.
Then, following the steps of \cite[Proposition 3.2]{sun2011bealeJMPA} we can show that  
\begin{equation}
    \div \bw \in L^\infty(0,T;L^2(\Omega)) \cap L^2(0,T;H^1(\Omega)) \hookrightarrow L^7(0,T;H^{\frac{2-\varepsilon}7}(\Omega)),
\end{equation}
where the last embedding holds for all $\varepsilon>0$ by using the interpolation of $L^p$ spaces~\cite[Theorem 5.1.2]{bergh2012interpolation}. 
We then use the embedding \cite[Theorem 3.3.1]{triebel1983theory} 
\begin{equation}
    H^{\frac{2-\varepsilon}7}(\Omega) \cong F^{\frac{2-\varepsilon}7}_{2,2} (\Omega) \hookrightarrow F^{\frac{-11-\varepsilon}{14}}_{7,1} (\Omega).
\end{equation}
For the equivalence between the fractional Sobolev space and the Triebel--Lizorkin space above, we refer the reader to \cite[Section 2.3.5]{triebel1983theory}. Putting this together with estimate~\eqref{v_velocity_ineq} (using once again the embedding result \cite[Theorem 3.3.1]{triebel1983theory}), we obtain that $\div \bu \in L^7(0.T;F^{\frac{-11-\varepsilon}{14}}_{7,1}(\Omega))$.

Then, making use of the triangle inequality, we can ensure \[\label{triangle_ineq_rhot}\varrho_t = -\varrho \div \bu - \nabla \varrho \cdot \bu \in  L^7(0.T;F^{\frac{-11-\varepsilon}{14}}_{7,1}(\Omega)),\]
 with the help of H\"older's inequality and the regularities \eqref{cond_reg_crit}, \eqref{cond_reg_crit_varrho}.

We have so far shown that $\varrho \in W^{1,7}(0,T;F^{\frac{-11-\varepsilon}{14}}_{7,1}(\Omega)) \cap L^7(0,T;W^{1,7})$, then by the time-trace result \cite[Chapter III, Theorem 4.10.2]{amann1995linear} and real interpolation properties of Triebel--Lizorkin spaces~\cite[Section 3.3.6 \& Theorem 2.4.2]{triebel1983theory}
\begin{equation}
    \varrho \in \operatorname{BUC}([0,T]; B_{7,7}^{\frac{73-\varepsilon}{98}}(\Omega)) \cong \operatorname{BUC}([0,T]; W^{\frac{73-\varepsilon}{98},7}(\Omega))
\end{equation}
where $\operatorname{BUC}$ stands for the space of bounded uniformly continuous functions, and the last space equivalence is given by \cite[Section2.2.2/18]{triebel1983theory}. 
Finally, using embedding results for fractional Sobolev spaces~\cite[Theorem 8.2]{DiNezza2012}, we obtain that $\varrho \in C([0,T]\times \overline{\Omega})$ (and so is $1/\varrho$) which is the needed regularity to apply the desired $L^p$--$L^q$ framework.

We mention here that results due to Krylov~\cite{krylov2007parabolicelliptic,krylov2007parabolic} suggest that the condition on the continuity of the coefficient $\frac1\varrho$ could be relaxed to allow for $VMO$ coefficients. We recall that the $VMO$ space is the closure of continuous functions in $BMO$. 

\subsection*{Source term regularity}
Let $T>0$. We estimate the source term norm $\|F\|_{L^p(0,T;L^q(\Omega))}$. Recall that the right-hand side term of the equation for $\bw$ is given by 
\begin{equation}
    F= -\bu\cdot\nabla \bu - L^{-1} \nabla (\pdt p( \varrho) ) + g,
\end{equation}
where $\pdt \bv = L^{-1} \nabla (\pdt \varrho )$.
To bound the different components of $F$, we proceed as follows. First, for arbitrary $\varepsilon>0$, we have
\begin{equation}\label{eq:estimate_rhs_1}
    \begin{aligned}
    \|\bu\cdot\nabla \bu\|_{L^p(0,T;L^q(\Omega))}^p \leq & \|\bu\|_{L^\infty([0,T]\times\Omega)}^p \|\nabla \bu\|_{L^p(0,T;L^q(\Omega))}^p\\
    \leq & C(\varepsilon) \|u\|_{L^\infty(\Omega)}^{2p} + \varepsilon \|\bw\|^{2p}_{L^p(0,T;W^{1,q}(\Omega))} +\varepsilon \|\bv\|^{2p}_{L^p(0,T;W^{1,q}(\Omega))}
    \\
    \leq & C(\varepsilon, \bar \varrho, \bar \bu)  + \varepsilon \|\bw\|^{2p}_{L^p(0,T;W^{1,q}(\Omega))}
    \end{aligned}
\end{equation}
where we have used the velocity decomposition: $\bu = \bv + \bw$ and the estimate for the Lam\'e system~\eqref{ineq:weak_ell_estimate}.
Furthermore, we use that the interpolation theory result~\cite[Theorem 6.4.5]{bergh2012interpolation}:
\begin{equation}\label{interpolation_result}
W^{1,q}(\Omega) = (L^q(\Omega),W^{2,q}(\Omega))_{[1/2]}
\end{equation}
which ensures that 
\begin{equation}
\begin{aligned}\label{eq:estimate_rhs_1/2}
\|\bw\|^{2p}_{L^p(0,T;W^{1,q}(\Omega))} \lesssim & \|\bw\|^{p}_{L^p(0,T;L^{q}(\Omega))} \|\bw\|^{p}_{L^p(0,T;W^{2,q}(\Omega))} \\
\lesssim & \|\bw\|_{L^\infty([0,T]\times\Omega)}^p  \|\bw \|_{L^p(0,T;W^{2,q}(\Omega))}^p\\
    \lesssim & (\|\bu\|_{L^\infty([0,T]\times\Omega)} + \|\bv\|_{L^\infty([0,T]\times\Omega)})^p \|\bw \|_{L^p(0,T;W^{2,q}(\Omega))(\Omega)}^p \\
    \lesssim& (\bar \bu + \|\bv\|_{L^\infty(0,T;W^{1,q}(\Omega))})^p \|\bw \|_{L^p(0,T;W^{2,q}(\Omega))}^p
    \\
    \lesssim& (\bar \bu + \bar\varrho)^p \|\bw \|_{L^p(0,T;W^{2,q}(\Omega))}^p
\end{aligned}
\end{equation}
Thanks to the small parameter $\varepsilon$ in \eqref{eq:estimate_rhs_1}, we will be able to absorb the norm $\|\bw \|_{L^p(0,T;W^{2,q}(\Omega))}$ by setting $\varepsilon$ small enough after \eqref{w_velocity_first_ineq}.

For the second term, we use a trick inspired by \cite[Proposition 3.2]{sun2011bealeJMPA} to remove the need to have control over the time derivative $\pdt p(\varrho)$, injecting instead the renormalized density equation; see e.g.,~\cite[Section B.1.1]{feireisl2022mathematics}:
\begin{equation}
    \pdt p(\varrho) + \div_x(p(\varrho) \bu) + \left(\varrho p'(\varrho) - p(\varrho)\right) \div_x \bu = 0.
\end{equation}
Thus using the bounds for the Lam\'e system from Lemma~\ref{lemma:Lamme_Lp_BMO}, we obtain that:
\begin{equation}\label{eq:estimate_rhs_2}
    \begin{array}{rl}
          \|L^{-1} \nabla  \div_x(p(\varrho) \bu) \|_{L^q(\Omega)} \lesssim  & \|p(\varrho) \bu\|_{L^q(\Omega)} \lesssim \|\varrho\|_{L^\infty}^\gamma \|\bu\|_{L^\infty(\Omega)} ,\\
          \| L^{-1} \nabla \left[\left(\varrho p'(\varrho) - p(\varrho)\right) \div_x \bu \right]\|_{L^q(\Omega)} \leq & \| L^{-1} \nabla \left[\left(\varrho p'(\varrho) - p(\varrho)\right) \div_x \bu \right]\|_{W^{1,q}(\Omega)} \\
          \lesssim & C(\bar \varrho, \bar \bu, \gamma ) \left(\|\nabla \bw \|_{L^q(\Omega)} + \|\nabla \bv \|_{L^q(\Omega)}\|  \right)\\
          \lesssim &  C(\bar \varrho, \bar \bu, \gamma ) \left(\|\nabla \bw \|_{L^q(\Omega)} + 1  \right),
    \end{array}
\end{equation}
for almost every time $t\in(0,T)$, where in the last line we used the estimate \eqref{ineq:weak_ell_estimate} for the velocity $\bv$. We also use the interpolation result~\eqref{interpolation_result} which ensures that 
\begin{equation}\label{eq:estimate_rhs_3}
\begin{aligned}
    \|\nabla \bw \|_{L^q(\Omega)} \leq & \| \bw \|_{L^q(\Omega)}^{1/2} \|\bw \|_{W^{2,q}(\Omega)}^{1/2} \\\leq & C(\varepsilon) \|\bw\|_{L^\infty(\Omega)} + \varepsilon \|\bw \|_{W^{2,q}(\Omega)}\\
    \leq& C(\varepsilon) (\|\bu\|_{L^\infty(\Omega)} + \|\bv\|_{L^\infty(\Omega)}) + \varepsilon \|\bw \|_{W^{2,q}(\Omega)} \\
    \leq& C(\varepsilon) (\bar \bu + \|\bv\|_{W^{1,q}(\Omega)}) + \varepsilon \|\bw \|_{W^{2,q}(\Omega)}
    \\
    \leq& C(\varepsilon) (\bar \bu + \bar\varrho) + \varepsilon \|\bw \|_{W^{2,q}(\Omega)}.
\end{aligned}
\end{equation}
Combining, \eqref{eq:estimate_rhs_2} and \eqref{eq:estimate_rhs_3} we find that 
\begin{equation}\label{eq:estimate_rhs_4}
    \begin{array}{rl}
          \|L^{-1} \nabla  \div_x(p(\varrho) \bu) \|_{L^p(0,T;L^q(\Omega))}^p \leq  & C(\bar \varrho, \bar \bu, \gamma, p, q, T)\\
          \| L^{-1} \nabla \left[\left(\varrho p'(\varrho) - p(\varrho)\right) \div_x \bu \right]\|_{L^p(0,T;L^q(\Omega))}^p \leq & C(\bar \varrho, \bar \bu, \gamma, p, q, T) \left(\varepsilon\| \bw \|_{L^p(0,T; W^{2,q}(\Omega))}^p + 1  \right),
    \end{array}
\end{equation}

Now using the $L^p-L^q$ existence result from \cite[Theorem 2.3]{denk2007optimal}, we obtain the following estimate for $\bw$:
	\begin{equation}\label{w_velocity_first_ineq}
		\begin{aligned}
			\int_0^T \big[\| \pdt \bw \|^p_{L^q(\Omega)} + & \|\bw\|_{W^{2,q}(\Omega)}^p \big]\dt   \\ \leq & \Lambda(\bar\varrho,\bar\bu,r_0, p, q,T)  \Bigg[  \int_0^T \|F\|^p_{L^q(\Omega)} +  \| \bu_B\|^p_{W^{2-1/q,q}(\partial \Omega)} \dt+ \\
            & \hphantom{\Lambda(\bar\varrho,\bar\bu,r_0, p, q,T)\Lambda(\bar\varrho,\bar\bu)}+ 
             \|\bu_0 - L^{-1}(\nabla p(\varrho_0))\|_{B^{2(1-1/p)}_{qp}(\Omega)}^p 
            \Bigg],
		\end{aligned}
	\end{equation}
where we stressed the dependence of $\Lambda$ on the data of the problem as well as on the magnitude of the coefficient $\frac 1 \varrho$; see, e.g., \cite{krylov2007parabolicelliptic,krylov2007parabolic}. 
The term $\int_0^T \|F\|^p_{L^q(\Omega)}$ is estimated using the bounds \eqref{eq:estimate_rhs_1}, \eqref{eq:estimate_rhs_1/2}, and \eqref{eq:estimate_rhs_4} and by choosing $\varepsilon$ small enough so that the $\varepsilon$-terms are absorbed by the left-hand side. Thus, we finish the proof of \eqref{w_velocity_ineq}.

Note, additionally, that combining \eqref{eq:estimate_rhs_4} with the fact that 
 $$\pdt \bv = L^{-1} \nabla (\pdt \varrho ) = L^{-1} (\div_x(p(\varrho) \bu) + \left(\varrho p'(\varrho) - p(\varrho)\right)),$$ and exploiting the Lam\'e system estimate~\eqref{ineq:ultra_weak_ell_estimate} yields inequality~\eqref{v_velocity_time_der}.
\end{proof}

\section{Estimate on the density}

Note that Lemma~\ref{lemma:estimates_velocity} yields control over $\bw$ in $L^p(0,T;W^{2,q}(\Omega)) \hookrightarrow L^1(0,T; W^{1,\infty} (\Omega))$ (since  $q>{\rm dim}=3$),  this regularity is employed for ensuring the existence of the characteristics (streamlines) necessary to solve the continuity equation. For $\bv$, while \eqref{v_velocity_ineq} does not offer sufficient control, the space $BMO(\Omega)$ is good enough combined with the Brezis--Gallou\"et--Wainger inequality of Lemma~\eqref{lemma:log_ineq} for the purposes of establishing the desired estimates for the mass density $\varrho$.

\begin{lemma}\label{lemma:mass_density}
    Under the assumptions of Lemma~\ref{lemma:estimates_velocity}, we obtain the following estimate on the mass density $\varrho$:
    \begin{equation}\label{rho_estimate}
        \|\varrho\|_{L^\infty(0,T;W^{1,q}(\Omega))}^q \leq \Gamma(\bar\bu, \varrho_0, \bu_0, \bu_B,\bar r, \bar G, p, q,T),
    \end{equation} 
    where $\Gamma$ is an increasing function of the arguments but which stays bounded for bounded arguments.
\end{lemma}
\begin{proof}
The starting point of our proof will be \cite[Lemma 5.2]{meliani2025p}. Indeed, Assumptions~\ref{assu:data} and \ref{assu:domain} are sufficient for the computations of \cite[Section 5.1]{meliani2025p} to be valid. The proof in \cite{meliani2025p} is based on an idea of \cite{fiszdon1983initial} to use the natural coordinate system near the boundary formed of the orthonormal vectors $n$ (outward normal) and $\{\tau_i\}_{1 \leq i \leq d-1}$ (tangential unit vectors). This allows to split the gradient along the boundary in order to control the boundary term in \cite[p. 13]{meliani2025p}:
\begin{equation}
	\begin{aligned}\label{ineq:density}
\ddt \|D^\alpha\varrho\|_{L^q(\Omega)}^q \leq &  q  \|D^\alpha \bu\|_{L^\infty(\Omega)} \|\nabla_x \varrho\|_{L^q(\Omega)} \|D^\alpha \varrho\|_{L^q(\Omega)}^{q-1}  - \int_{\partial\Omega} \bu \cdot n (|D^\alpha \varrho|^q) \dsig \\ &+ \|\div_x \bu\|_{L^\infty(\Omega)} \|D^\alpha\varrho\|_{L^q(\Omega)}^{q} \\ &
+  \|\bu\|_{W^{2,q}(\Omega)} \|D^\alpha\varrho\|_{L^q(\Omega)}^{q-1} \|\varrho\|_{L^\infty(\Omega)}. 
	\end{aligned}
\end{equation}
Above $\alpha$ is a multi-index with $|\alpha|\leq 1$, thus $D^\alpha$ are the different first-order spatial derivatives.

For the convenience of the reader we recall here the details of bounding the derivative term $|D^\alpha \varrho|$ on the boundary, particularly on the inflow part of the boundary. Note that on $\partial\Omega \setminus \Gamma_\textup{in}$, the normal component of the velocity is nonnegative $\bu\cdot n\geq 0$, such that:
\[- \int_{\partial\Omega \setminus \Gamma_\textup{in}} \bu \cdot n (|D^\alpha \varrho|^q) \dsig \leq 0.\]

\subsection*{Controlling the boundary term.}

In the tangential directions $\{\tau_i\}_{1 \leq i \leq d-1}$, control of the derivative $D^\alpha\varrho$ on $\Gamma_\textup{in}$ is directly given by the regularity of the boundary condition $\varrho_B$ and that of the velocity field $\bu$:
\begin{equation}\label{ineq:control_tan}
	\begin{aligned}
		- \int^t_0 \int_{\Gamma_{\textup{in}}}  \bu \cdot n \left|\frac{\partial}{\partial \tau_i}\varrho\right|^q \dsig \dt   = &- \int^t_0 \int_{\Gamma_{\textup{in}}} \bu \cdot n \left|\frac{\partial}{\partial \tau_i}\varrho_B\right|^q \dsig \dt \\ \leq & \| \bu_B \cdot n\|_{L^\infty(\Gamma_{\textup{in}})} \|\varrho_B\|_{L^q(0,T;W^{1,q}(\Gamma_{\textup{in}}))}^q,
	\end{aligned}
\end{equation}
where $i\in \{1, \ldots ,d-1\}.$

To establish an estimate in the normal direction to the boundary, we lift the boundary condition $\varrho_B$ to a function on $\Omega$ (which we still denote $\varrho_B$). 
Let $\tilde{\varrho} = \varrho - \varrho_B$, then $\tilde{\varrho}$, verifies:
\begin{equation}\label{eq:homogeneous_prob}
	\pdt \tilde\varrho + \bu \cdot \nabla_x \tilde\varrho + \div_x \bu \tilde\varrho = -\pdt \varrho_B - \bu \cdot \nabla_x \varrho_B - \div_x \bu \varrho_B.
\end{equation}
Formally, evaluating \eqref{eq:homogeneous_prob} on $\Gamma_{\textup{in}}$, we obtain:
\begin{equation}\bu \cdot \nabla_x \tilde\varrho = -\pdt \varrho_B - \div_x (\varrho_B \bu)
\end{equation}
owing to  $\tilde\varrho |_{\Gamma_{\textup{in}}} = 0$. We further simplify by noticing that the tangential component of $\nabla \tilde{\varrho} $ is equal to $0$, such that 
\[\bu \cdot \nabla_x \tilde\varrho = (\bu \cdot n) (\nabla_x \tilde\varrho \cdot n) = (\bu \cdot n) \frac{\partial}{\partial n} \tilde\varrho,\]
this provides control on the normal derivatives as long as $\bu\cdot n$ is bounded away from $0$ on $\Gamma_{\textup{in}}$. Indeed, we can test the resulting equation on the boundary by $\left|\frac{\partial}{\partial n}\tilde\varrho\right|^{q-2} \frac{\partial}{\partial n}\tilde\varrho$ and integrate over the inflow part of the boundary ($\Gamma_{\textup{in}}$) and in time on $(0,t)$, such that  we obtain
\begin{equation}
	\begin{aligned}
		\int^t_0 \int_{\Gamma_{\textup{in}}} (- \bu \cdot n) \left|\frac{\partial}{\partial n}\tilde\varrho\right|^q \dsig \dt =   \int^t_0 \int_{\Gamma_{\textup{in}}} \left( \pdt \varrho_B + \varrho_B \div_x \bu + \bu \cdot \nabla_x \varrho_B \right) \left|\frac{\partial}{\partial n}\tilde\varrho\right|^{q-2} \frac{\partial}{\partial n}\tilde\varrho \dsig \dt&\\
		\leq  C(\varepsilon) \int^t_0 \int_{\Gamma_{\textup{in}}} \left| \pdt \varrho_B + \varrho_B \div_x \bu + \bu \cdot \nabla_x \varrho_B \right|^q \dsig \dt+ \varepsilon \int^t_0 \int_{\Gamma_{\textup{in}}}  \left|\frac{\partial}{\partial n}\tilde\varrho\right|^{q}  \dsig \dt,&\\
	\end{aligned}  
\end{equation}
where the last line was obtained using Young's inequality with exponents $q$ and $\frac{q}{q-1}$. The constant $\varepsilon>0$ is arbitrary, $C(\varepsilon)>0$ is a $\varepsilon$-dependent constant. Now, we use the lower bound of the velocity on the inflow part of the boundary $c>0$ as expressed in \eqref{data_velocity}:
\begin{equation}
	\begin{aligned}
		c \int^t_0 \int_{\Gamma_{\textup{in}}} \left|\frac{\partial}{\partial n}\tilde\varrho\right|^q \dsig \dt
		\leq & C(\varepsilon) \int^t_0 \int_{\Gamma_{\textup{in}}} \left| \pdt \varrho_B + \varrho_B \div_x \bu + \bu \cdot \nabla_x \varrho_B \right|^q \dsig \dt+ \varepsilon \int^t_0 \int_{\Gamma_{\textup{in}}}  \left|\frac{\partial}{\partial n}\tilde\varrho\right|^{q}  \dsig \dt,\\
	\end{aligned}  
\end{equation}
thus choosing $\varepsilon$ smaller than the constant $c$ (e.g., $\frac c2$), we obtain that 
\begin{equation}\label{est_boundary_in}
	\begin{aligned}
		\left(\int^t_0 \int_{\Gamma_{\textup{in}}} \left|\frac{\partial}{\partial n}\tilde\varrho\right|^q \dsig \dt\right)^{1/q} & \leq \left(\frac{2}{c}\right)^{1/q} \|\pdt \varrho_B + \bu \cdot \nabla_x\varrho_B + \div_x \bu \varrho_B \|_{L^q(0,t;L^q(\Gamma_{\textup{in}}))} \\
		&\lesssim \|\pdt \varrho_B\|_{L^q(0,t;L^q(\Gamma_{\textup{in}}))} + \|\bu_B\|_{L^\infty(\Gamma_{\textup{in}})} \|\varrho_B\|_{L^q(0,t;W^{1,q}(\Gamma_{\textup{in}}))} \\ 
		& \qquad + \|\div_x \bu \varrho_B \|_{L^q(0,t;L^q(\Gamma_{\textup{in}}))}.
	\end{aligned}
\end{equation}

\subsubsection*{Estimating $ \|\div_x \bu \varrho_B \|_{L^q(0,t;L^q(\Gamma_{\textup{in}}))}$.}
We present the estimate procedure for the more involved case $p<q$. Indeed, the estimate \eqref{eq:estimate_divv_rho} below holds in a more elementary way for the case $p\geq q$.

First, using trace spaces of Sobolev spaces, we identify the space to which the restriction of $\div_x \bu$ to $\Gamma_{\textup{in}}$ belongs. Then, owing to \cite[Chapter III, Theorem 4.10.2]{amann1995linear} (see also \cite[Corollary 3.12.3]{bergh2012interpolation}), we know that 
\begin{equation}\label{divx_embedding}
	\div_x \bu \in C([0,t]; B^{1-\frac2p -\frac1q}_{qp} (\Gamma_{\textup{in}})) \cap L^p(0,t; W^{1-\frac1q,q} (\Gamma_{\textup{in}})) \hookrightarrow L^q (0,t; B^{1-\frac2p + \frac 1q }_{qq'} (\Gamma_{\textup{in}})),
\end{equation}
and 
$$\varrho_B \in C([0,t]; B^{1-\frac 1q}_{qq} (\Gamma_{\textup{in}})) \cap L^{q}(0,T; W^{1, q}(\Gamma_\textup{in})),$$
where for the embedding in \eqref{divx_embedding} we used real interpolation with the parameters $0<p/q<1$ and $q'$, where $q'$ is the conjugate H\"older exponent of $q$. (Note that interpolation of $L^p$ spaces (time component) is simply obtained using H\"older's inequality).

We note that $q>d\geq 2 >q'$ and that $1-\frac2p + \frac 1q  \geq 0$ due to Assumption~\ref{assu:data}. Thus using standard properties of Besov spaces~\cite[Section 2.3.2, Propostition 2]{triebel1983theory} (see also the remark after \cite[Proposition 3.2.4]{triebel1983theory} and \cite[Proposition 3.3.1]{triebel1983theory}), we know the following embeddings hold:
$$B^{1-\frac2p + \frac 1q}_{qq'} (\Gamma_{\textup{in}}) \hookrightarrow B^{0}_{qq'}  (\Gamma_{\textup{in}}) \hookrightarrow F^{0}_{q2}  (\Gamma_{\textup{in}}) = L^q (\Gamma_{\textup{in}}).$$

On the other hand, we also have the embedding
$B^{1-\frac 1q }_{qq} (\Gamma_{\textup{in}}) \hookrightarrow C(\Gamma_{\textup{in}}) \hookrightarrow L^\infty(\Gamma_{\textup{in}}),$ due to the condition $q>d$.

Putting these embeddings together, we can estimate the last term in \eqref{est_boundary_in} as 
\begin{equation}\label{eq:estimate_divv_rho}
	\|\div_x \bu \varrho_B \|_{L^q(0,t;L^q(\Gamma_{\textup{in}}))}\lesssim  \|\div_x \bu  \|_{L^q(0,t;L^q(\Gamma_{\textup{in}}))}\|\varrho_B\|_{L^\infty(0,t;L^\infty(\Gamma_{\textup{in}}))}.
\end{equation}	

\color{black}

Now, going back to \eqref{ineq:density}, and combining it with estimates~\eqref{est_boundary_in} and \eqref{eq:estimate_divv_rho}, we obtain after integration on $(0,T)$ 
\begin{equation}
	\begin{aligned}
		 |\varrho|_{W^{1,q}(\Omega)}^q \lesssim \,& |\varrho_0|_{W^{1,q}(\Omega)}^q + \int_0^T |\bu|_{W^{1,\infty}(\Omega)} |\varrho|_{W^{1,q}(\Omega)}^q \dt\\ & + \int_0^T \|\div_x \bu\|_{L^\infty(\Omega)} |\varrho|_{W^{1,q}(\Omega)}^{q} \dt \\ &
		+  \int_0^T\|\bu\|_{W^{2,q}(\Omega)} |\varrho|_{W^{1,q}(\Omega)}^{q-1} \|\varrho\|_{L^\infty(\Omega)} \dt \\& 	+ 
		\| \bu_B \cdot n\|_{L^\infty(\Gamma_{\textup{in}})} \|\varrho_B\|_{L^q(0,T;L^q(\Gamma_{\textup{in}}))}^q\\&+  \|\pdt \varrho_B\|^q_{L^q(0,t;L^q(\Gamma_{\textup{in}}))} + \|\bu_B\|_{L^\infty(\Gamma_{\textup{in}})}^q \|\varrho_B\|_{L^q(0,t;W^{1,q}(\Gamma_{\textup{in}}))}^q \\ &+\|\div_x \bu  \|^q_{L^q(0,t;L^q(\Gamma_{\textup{in}}))}\|\varrho_B\|^q_{L^\infty(0,t;L^\infty(\Gamma_{\textup{in}}))}. 		
	\end{aligned}
\end{equation}

Now we use the decomposition $\bu = \bw + \bv$, and that according to Lemma~\ref{lemma:Lamme_Lp_BMO}:
\begin{equation}
    \|\bv\|_{W^{2,q}} \lesssim \|\nabla \varrho \|_{L^q(\Omega)}.
\end{equation}
Note that this also gives us control over the boundary term:
\begin{equation}
   \|\div_x \bv  \|^q_{L^q(0,t;L^q(\Gamma_{\textup{in}}))}  \lesssim \|\bv\|_{L^q(0,T;W^{2,q})}^q \lesssim \int_0^T \|\nabla \varrho \|_{L^q(\Omega)}^q \dt
\end{equation}
Furthermore, we know, through Lemma~\ref{lemma:log_ineq}, that 
\begin{equation}
    \|\div_x \bv\|_{L^\infty(\Omega)} \leq |\bv|_{W^{1,\infty}(\Omega)}  \lesssim C( 1+ \bar\varrho \ln(e+ \|\nabla \varrho\|_{L^q(\Omega)})),
\end{equation}
where $\bar\varrho$ is as in \eqref{cond_reg_crit_varrho}.

\noindent Thus, putting the above inequalities together, and taking advantage of the estimates of Lemma~\ref{lemma:estimates_velocity}, we obtain:
\begin{equation}
        \begin{aligned}
		 |\varrho|_{W^{1,q}(\Omega)}^q \lesssim \,& C_1+  C_2\int_0^T |\varrho|_{W^{1,q}(\Omega)}^{q} \dt + C_3 \int_0^T |\varrho|_{W^{1,q}(\Omega)}^{q} \ln(e+ |\varrho|_{W^{1,q}(\Omega)}) \dt  \\
         \lesssim \,& C_1+  (C_2 + C_3) \int_0^T |\varrho|_{W^{1,q}(\Omega)}^{q} \ln(e+ |\varrho|_{W^{1,q}(\Omega)}) \dt 
	\end{aligned}
\end{equation}
where the constants $C_1$, $C_2$ depend on the problem data as well as on $\bar \bu$, $\bar r$, $p$, $q$, and $T$. We have used, in the last line, that $\ln(e+ |\varrho|_{W^{1,q}(\Omega)})\geq 1$.

We then apply Gronwall's Lemma along the lines of \cite{brezis1980nonlinear} and obtain boundedness of the density as long as $T<\infty$. 
\end{proof}

Having obtained a bound on the $\varrho$ in $L^\infty(0,T;W^{1,q}(\Omega))$, we use Lemma~\ref{lemma:Lamme_Lp_BMO} to improve the component $\bv$ of the velocity. Indeed, thanks to \eqref{ineq:strong_ell_estimate}, $\bv \in L^\infty(0,T;W^{2,q}(\Omega))$ and 
\begin{equation}\label{bootstrap_estimate_bv}
    \| \bv \|_{L^\infty(0,T;W^{2,q}(\Omega))} \leq \|\varrho\|_{L^\infty(0,T;W^{1,q}(\Omega))}.
\end{equation}
Putting together estimates~\eqref{w_velocity_ineq}, \eqref{v_velocity_time_der} and \eqref{bootstrap_estimate_bv}, leads to the conclusion that the velocity $\bu$ is bounded in the space 
\[L^p(0,T; W^{2,q}(\Omega)) \cap W^{1,p}(0,T;L^q(\Omega)),\]
which embeds in the space $C([0,T]; B^{2(1-1/p)}_{qp} (\Omega))$; see, e.g., \cite[Chapter III, Theorem 4.10.2]{amann1995linear}.

Finally, by writing:
\begin{equation}
\pdt \varrho = -  \varrho \div_x \bu + \nabla\varrho \cdot \bu,
\end{equation}
and using the fact that $\bu \in C([0,T]; B^{2(1-1/p)}_{qp} (\Omega))$, with $p>2$ (see conditions on $p$ in Section~\ref{sec:Main_result}), we readily show that $\pdt \varrho$ is bounded in $L^\infty(W^{1,q}(\Omega))$.

The proof of Theorem~\ref{thm:main} is thus completed.
\section*{Conclusions and perspectives}

In this work, we have established a first blow-up criterion for the compressible Navier--Stokes system with inflow-outflow boundary conditions. To overcome the challenges posed by the inflow-outflow setting, our approach made use of the $L^p$--$L^q$ framework as a way to obtain the necessary estimates on the velocity ($\div \bu$ `almost' in $L^1(L^\infty(\Omega))$) which were then combined with density estimates taking advantage of logarithmic Sobolev inequalities (Brezis--Gallouet--Wainger inequality). 

Our approach has one major drawback, that is, we need to ensure that $\frac1\varrho$ remains continuous on the closure $[0,T] \times \overline \Omega$. We achieved this by assuming boundedness of the gradient  of the density (i.e. $\|\nabla \varrho\|_{L^\infty(L^7)} $). We believe that future work should be tasked with relaxing this condition, eventually by only imposing boundedness of the  gradient of the density on the inflow part of the boundary $\Gamma_\textup{in}$. Achieving such a result would have the great advantage of unifying the theory of blow-up criteria for both inflow and no-inflow cases. 

Other questions of interest, concern extending the result to the Navier--Stokes--Fourier model and to equations of magnetohydrodynamics. 

\section*{Declarations}
The research of A.A. is performed under the auspices of GNFM--INdAM.\\
The research was conducted while M.M. was a postdoctoral researcher at the Institute of Mathematics of the Academy of Sciences of the Czech Republic. His work was supported by the Czech Sciences Foundation (GA\v CR), Grant Agreement 24-11034S. The Institute of Mathematics of the Academy of Sciences of the Czech Republic is supported by
RVO:67985840. \\[0mm]

\noindent The authors state that there is no conflict of interest.\\[0mm]

\noindent No data are associated with this article

\section*{Acknowledgements}
 The authors are grateful to Prof. Eduard Feireisl (AV\v CR) for helpful discussions during the preparation of this work and for suggestions leading to the simplification of the blow-up criterion.

\bibliography{references}{}
\bibliographystyle{siam} 
\end{document}

%% file: preamble.tex
\usepackage{smart-eqn}
\newcommand{\bu}{\boldsymbol{u}}
\newcommand{\bv}{\boldsymbol{v}}
\newcommand{\bw}{\boldsymbol{w}}

\usepackage{enumitem}
\usepackage{hyperref}
\renewcommand{\div}{\operatorname{div}}

\newcommand{\R}{\mathbb{R}}
\renewcommand{\SS}{\mathbb{S}}
\newcommand{\dsig}{\,\textup{d}\sigma}
\newcommand{\pdt}{\partial_t}
\newcommand{\ddt}{\frac{\textup{d}}{\textup{d}t}}

\usepackage{autonum}
\usepackage{xcolor}
\usepackage{amsthm}
\usepackage{amssymb}
\usepackage{mathtools}
\usepackage{xcolor}
\usepackage{soul}

\usepackage{comment}
\usepackage[margin=.99in]{geometry}
\newtheorem{theorem}{Theorem}
\newtheorem{lemma}{Lemma}
\newtheorem{proposition}{Proposition}

\newtheorem*{assumption*}{Assumptions}
\newtheorem{assumption}{Assumption}

\newtheorem{remark}{Remark}
\numberwithin{lemma}{section}
\numberwithin{proposition}{section}
\numberwithin{theorem}{section}
\numberwithin{equation}{section}
\usepackage{xcolor}
\usepackage{framed}
\usepackage{tocloft}

\usepackage{graphicx} 

\newcommand{\bFormula}[1]{
	\begin{equation} \label{#1}}
	\newcommand{\eF}{\end{equation}}

\newcommand{\vu}{\vc{u}}

\newcommand{\vc}[1]{{\bf #1}}

\newcommand{\dt}{\,{\rm d} t }

\def\softd{{\leavevmode\setbox1=\hbox{d}%
		\hbox to 1.05\wd1{d\kern-0.4ex{\char039}\hss}}}